\documentclass[11pt]{amsart}

\usepackage{graphics}
\usepackage[dvips]{epsfig}
\usepackage{psfrag}
\usepackage{pinlabel}
\usepackage{amsmath}
\usepackage{amsfonts}
\usepackage{latexsym}
\usepackage{amssymb}
\usepackage[usenames]{color}
\usepackage{amsthm}
\usepackage[utf8]{inputenc}

\input xy
\xyoption{all}

\theoremstyle{plain}
\newtheorem{defn}{Definition}[section]
\newtheorem{Thm}[defn]{Theorem}
\newtheorem{Prop}[defn]{Proposition}

\theoremstyle{remark}

\DeclareMathAlphabet{\mathdj}{U}{msb}{m}{n}

\begin{document}
\title{A note on exact Lagrangian cobordisms with disconnected Legendrian ends}
\author{Baptiste Chantraine}
\address{Universit\'e de Nantes, France.}
\email{baptiste.chantraine@univ-nantes.fr}
\begin{abstract}
 We provide in this note two relevant examples of Lagrangian cobordisms. The first one gives an example of two exact Lagrangian submanifolds which cannot be composed in an exact fashion. The second one is an example of an exact Lagrangian cobordism on which all primitive of the Liouville form is not constant on the negative end and such that the positive end is a stabilisation whereas the negative end admits augmentations. These examples emphasise point (i) of the definition of exact Lagrangian cobordisms in \cite{Ekhoka}. In order to provide such examples we construct Lagrangian immersions with single double points using an explicit model and interpret such Lagrangians as cobordisms from the Hopf link.
\end{abstract}
\maketitle

\section{Introduction}
\label{sec:introduction}

Let $\Lambda^-$ and $\Lambda^+$ be two compact Legendrian submanifolds of a co-oriented contact manifold $(M,\xi=\ker\alpha)$. Recall that an exact Lagrangian cobordism from $\Lambda^-$ to $\Lambda^+$ is a proper Lagrangian submanifold $L$ of $(\mathbb{R}\times M,d(e^t\alpha))$ satisfying:
\begin{enumerate}
\item $L\cap (-\infty,-T)\times M=(-\infty,-T)\times\Lambda^-$.\label{item:1}
\item $L\cap (T,\infty)\times M=(T,\infty)\times\Lambda^+$.\label{item:2}
\item There exists a function $f$ satisfying $e^t\alpha\vert_{L}=df$ and such that $f^-:=f\vert_{(-\infty,-T)\times\Lambda^-}$ and $f^+:=f\vert_{(T,\infty)\times\Lambda^+}$ are constant.\label{item:3}
\end{enumerate}

Basic properties of this relation were studied in \cite{chantraine_conc}. In \cite{Ekhoka} it is shown that such a cobordism induces a map $\phi_L:\mathcal{A}(\Lambda^+)\rightarrow \mathcal{A}(\Lambda^-)$ where $\mathcal{A}$ denotes the Chekanov-Eliashberg algebra as in \cite{Chekanov_DGA_Legendrian}, \cite{Ekholm_Contact_Homology} and \cite{LCHgeneral}. This map satisfies the functorial property that if $L_1$ and $L_2$ are exact Lagrangian cobordisms then $\phi_{L_1\circ L_2}=\phi_{L_1}\circ\phi_{L_2}$ in homology.

The aim of this note is to emphasise the importance of the third point in the definition (which is condition (i) of \cite[Definition 1.1]{Ekhoka}) in order for these results to be true. Note that from (\ref{item:1}) and (\ref{item:2}) of the definition one can see that the functions are locally constant at $\pm\infty$ and thus problems occur when the involved negative end is disconnected. A first problem that can occur is that if we remove the third condition exact cobordisms do not form a category (one cannot compose morphisms in an exact fashion). More dramatically, without such a condition one cannot prevent the existence of holomorphic curves as in Figure \ref{fig:holom} and thus the map $\phi_L$ is ill-defined. We provide explicit examples of those two phenomena in Section \ref{sec:expl-lagr-cobord}.

\begin{figure}[!h]
  \centering
  \includegraphics[height=3cm]{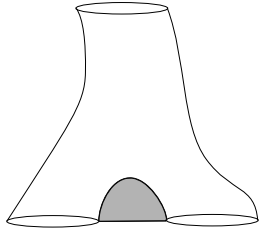}
  \caption{Holomorphic half disk on an exact Lagrangian cobordism.}
  \label{fig:holom}
\end{figure}

In what follows we consider only two contact manifolds: $(S^3,TS^3\cap i\cdot TS^3)\subset\mathbb{C}^2\simeq\mathbb{R}^4$ and $(\mathbb{R}^3,\ker(dz-ydx))$. We also fix a contactomorphism $f:\mathbb{R}^3\rightarrow S^3\setminus\{(0,\frac{\sqrt{2}}{2},0,\frac{\sqrt{2}}{2})\}$ as in \cite[Proposition 2.1.8]{Geiges_contact}.

The symplectisation of $S^3$ is naturally $\mathbb{R}^4\setminus\{0\}$ with the standard symplectic form $\omega_0$. The symplectisation of $\mathbb{R}^3$ is $\mathbb{R}\times\mathbb{R}^3$ with the symplectic form $d(e^t(dz-ydx))$ which is symplectomorphic to $T^*(\mathbb{R}^*_+\times\mathbb{R})$ with its standard symplectic form via
\begin{align}
  \label{eq:1}
  (t,x,y,z)\rightarrow (e^t,x,z,e^t\cdot y)=(q,s,p_1,p_2).
\end{align}

It will be important to remember that the second cotangent direction is expanding under the symplectomorphism \eqref{eq:1}. The contactomorphism $f$ induces a symplectomorphism $\widetilde{f}$ from $T^*(\mathbb{R}^*_+\times\mathbb{R})$ to $\mathbb{R}^4\setminus\{(0,t,0,t)\vert t\in\mathbb{R}_+^*\}$. In the following $(Y,\xi)$ will denote any of those two contact manifolds and $S(Y,\xi)$ will denote its symplectisation where it is understood that the symplectic form is the standard one under the appropriate identification. We denote by $\Lambda_0$ the trivial Legendrian knot parametrised in $\mathbb{R}^3$ by $(\cos\theta,-3\sin\theta\cos\theta,\sin^3\theta)$ and by $H$ the Hopf Legendrian link $\Lambda_0\cup (\Lambda_0+\varepsilon\frac{\partial}{\partial t})$. We also denote by $S^+$ and $S^-$ the positive and negative stabilisations of Legendrian links.

Among the examples we provide the most disturbing at first sight is a cobordism whose negative end is a stabilisation and whose positive end admits augmentations.

\begin{Prop}\label{sec:prop}
  There exists an oriented Lagrangian cobordism $L$ from $H$ to $S^+(S^-(\Lambda_0))$ such that the Liouville form $e^t\alpha$ is exact on $L$.
\end{Prop}

It is well known that in Chekanov-Eliashberg algebras of stabilisations the unit is a boundary.  However $H$ has augmentations (independently of the grading we choose for mixed chords) and thus the unit of its Chekanov algebra is not a boundary. This implies that the cobordism $L$ even though exact as Lagrangian submanifold cannot induce a DGA-map from $\mathcal{A}(S^+(S^-(\Lambda_0))$ to $\mathcal{A}(H)$. It will be obvious from the construction that $L$ does not satisfy the condition (\ref{item:3}) of the definition of exact Lagrangian cobordism.

In order to construct the explicit examples, we uses elementary cobordisms as defined in \cite{collarable} together with an additional elementary Lagrangian immersion provided by the following theorem: 

\begin{Thm}\label{thm:elementarymove}
  Suppose that the front projections of two Legendrian links $\Lambda^-$ and $\Lambda^+$ only differs by the local move of Figure \ref{fig:Imm} then there exists a Lagrangian immersion $L$ of a pair of pants in $\mathbb{R}\times\mathbb{R}^3$ with a single double point such that
  \begin{itemize}
\item $L\cap (-\infty,-T)\times M=(-\infty,-T)\times\Lambda^-$
  \item $L\cap (T,\infty)\times M=(T,\infty)\times\Lambda^+$
  \end{itemize}
\end{Thm}

\begin{figure}[!h]
  \centering
  \includegraphics[height=3cm]{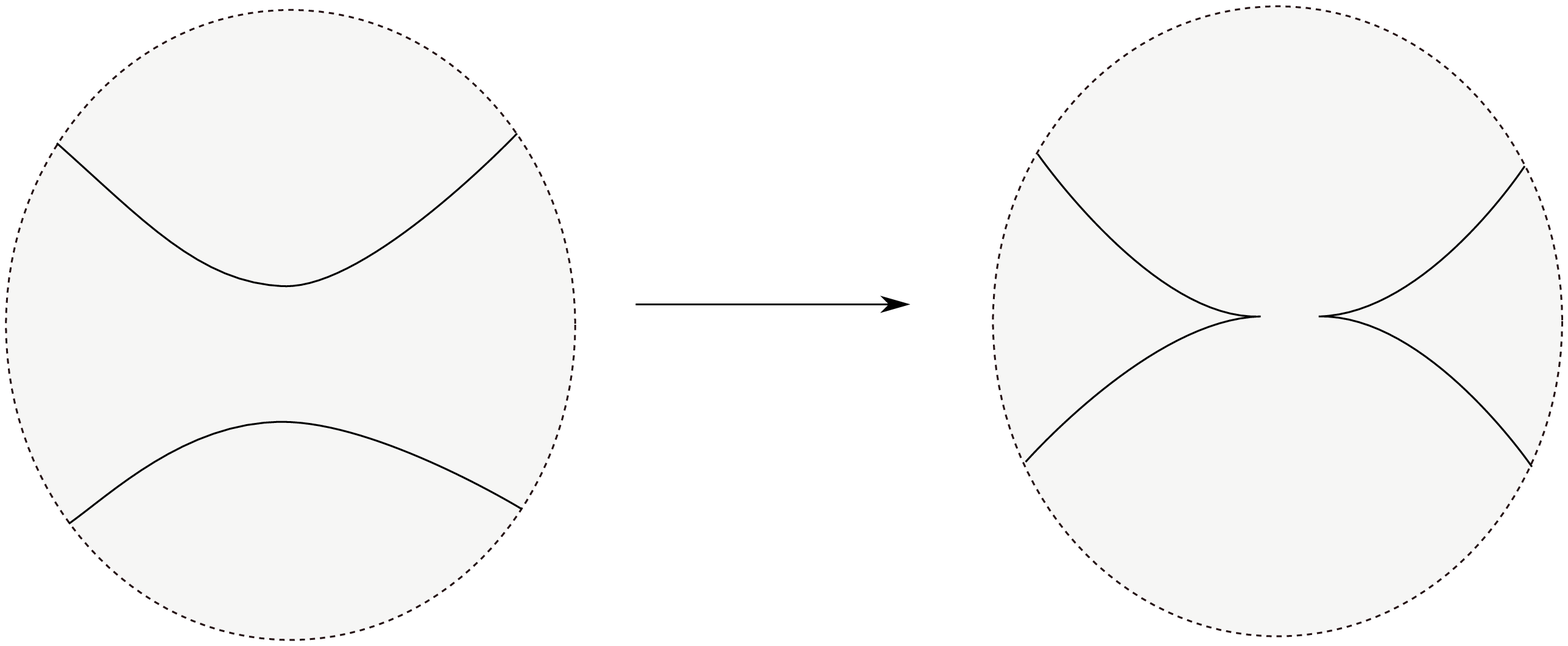}
\caption{Local move realisable by a Lagrangian immersion.}
  \label{fig:Imm}
\end{figure}

Note that the move of Figure \ref{fig:Imm} is the opposite move of move number $4$ of \cite[Figure 1]{collarable}. It can only be realised by an embedded Lagrangian cobordism in the direction opposite to the arrow. If one attempts to go in the other direction (as in Figure \ref{fig:Imm}) one creates a transverse double point as stated in Theorem \ref{thm:elementarymove}. 

In order to interpret Lagrangian immersions with a single double point and empty negative end with cobordisms from the Hopf link we use the following simple theorem:

 \begin{Thm}\label{thm:doublepoint}
   Let $i:\Sigma\rightarrow S(Y,\xi)$ be a Lagrangian immersion such that: 
   \begin{enumerate}
   \item $i$ has only one transverse double point $i(p_0)=i(p_1)=q$.
\item $i(\Sigma)\cap (T,\infty)\times Y=(T,\infty)\times \Lambda$ where $\Lambda$ is a Legendrian submanifold of $Y$.
\item $i(\Sigma)\cap (-\infty,-T)\times Y=\emptyset$
   \end{enumerate}
then there exists a Lagrangian cobordism $i':\Sigma'\rightarrow S(Y,\xi)$ where $\Sigma'=\Sigma\setminus \{p_0,p_1\}$ from the Legendrian Hopf link to $\Lambda$. Moreover if $i$ is exact then so is $i'$.
 \end{Thm}

In Section \ref{sec:lagr-immers-with} we give the proof of Theorem \ref{thm:doublepoint} and then construct the elementary cobordism of Theorem \ref{thm:elementarymove}. In Section \ref{sec:expl-lagr-cobord} we illustrate explicitly the necessity of (\ref{item:3}) of the definition of exact Lagrangian cobordism and prove Proposition \ref{sec:prop}. We also provide, similar to \cite[Question 8.10]{Ekhoka}, an example of a non-decomposable Lagrangian cobordism as defined in \cite{collarable}.

\textbf{Acknowledgements:} The existence of those cobordisms emanated from a discussion with Samuel Lisi whom I thank for leading to the construction of Theorem \ref{thm:doublepoint}. I also thank Paolo Ghiggini for explaining that the holomorphic curves of Figure \ref{fig:holom} can occur when negative ends are disconnected. Finally I thank Tobias Ekholm, Ko Honda et Tam\'as K\'alm\'an for reading and commenting a first draft of these constructions.

\section{Lagrangian immersions with a single double point.}
\label{sec:lagr-immers-with} We start by giving the proof of Theorem \ref{thm:doublepoint}.

\begin{proof}[Proof of Theorem \ref{thm:doublepoint}]
We assume that the target of $i$ is $\mathbb{R}^4$ (using $\widetilde{F}$ if necessary).

Since the Hamiltonian diffeomorphisms group of a connected symplectic manifold is transitive one can assume that $q=0$. Now a standard application of Moser's path method applied twice allows us to find a compactly supported Hamiltonian diffeomorphisms supported in a small Darboux ball $B=D^4_\varepsilon$ around $0$ such that $B\cap \phi\circ i(\Sigma)=B\cap(\mathbb{R}^2\cup i\cdot\mathbb{R}^2)$. In this case $\phi\circ i(\Sigma)\cap S^3_\varepsilon$ is the Legendrian Hopf link and  $B\cap \phi\circ i(\Sigma)\setminus\{0\}$ is the cylinder over $H$. Thus $\phi\circ i(\Sigma)\setminus\{0\}$ is a Lagrangian cobordism from $H$ to $\Lambda$. Furthermore Hamiltonian diffeomorphisms preserves the exactness properties. Since one can realise transverse intersection between isotropic submanifolds, $\phi\circ i(\Sigma)$ avoid the half-line $\{(t,0,t,0)\}$ and thus the proof is complete even when $Y=\mathbb{R}^3$.
 \end{proof}

We now turn our attention to Theorem \ref{thm:elementarymove}.  We give an explicit description of the elementary immersion which, together with Theorem \ref{thm:doublepoint}, will produce an example of a an exact Lagrangian cobordism where the end at $\infty$ is a stabilised knot whereas the end at $-\infty$ is the Hopf link (see Section \ref{sec:expl-lagr-cobord}).

\begin{proof}[Proof of Theorem \ref{thm:elementarymove}]
  We use generating families to produce our elementary immersion
  realising the local move of Figure \ref{fig:Imm}. We assume that the
  local picture are in a Darboux ball. We recall that we denote by
  $(q,s,p_1,p_2)$ the coordinates on
  $T^*(\mathbb{R}_+^*\times\mathbb{R})$

  The move is described by the generating family given by Figure
  \ref{fig:Gen_FamilyImm}. 

\begin{figure}[!h]
  \centering
  \includegraphics[height=3cm]{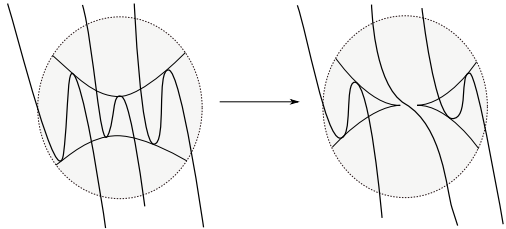}
  \caption{Generating family for the Lagrangian immersion.}
  \label{fig:Gen_FamilyImm}
\end{figure}

More explicitly, up to Legendrian isotopy, the left hand side of the picture is given by the
generating family $F_-(s,\eta)=\frac{\eta^3}{3}-\rho_-(s)\eta$ where
$\rho$ is a positive function such that $\rho_-^{\frac{3}{2}}$ is
smooth and $\rho_-'<0$ for $s<0$ and $\rho_-'>0$ for $s>0$ (this
ensure that the associated Legendrian sub-manifolds has only one Reeb
chord). Similarly the right hand side is given by a
generating family $F_+(s,\eta)=\frac{\eta^3}{3}-\rho_+(s)\eta$ where
$\rho_+$ is positive when $\vert s\vert>\frac{1}{2}$ and negative when
$\vert s\vert<\frac{1}{2}$, $\rho_+^{\frac{3}{2}}$ is smooth and $\rho_+'<0$ for $s<0$ and $\rho_+'>0$
for $s>0$. We also require that $\rho_-$ and $\rho_+$ coincide in a
neighbourhood $\pm 1$.

Let $\rho(q,s)$ be a smooth function such $\rho(q,s)=\rho_-(s)$ if
$q<\frac{1}{T}$ and $\rho(t,s)=\rho_+(s)$ if $q>T$. We assume that
$\frac{\partial \rho}{\partial s}$ is $0$ only when $s=0$. We claim
that $F(q,s,\eta)=q(\frac{\eta^3}{3}-\rho(q,s)\eta)$ is a generating
family for the desired immersion.

Explicitly the Lagrangian immersion
is: $$\Sigma=\{(q_0,s_0,p_1,p_2)\vert\exists \eta_0\text{ s.t. }
\eta_0^2=\rho(q_0,s_0),\; p_1=
\frac{\eta_0^3}{3}-\rho(q_0,s_0)\eta_0-q_0\frac{\partial\rho}{\partial
  q}\cdot\eta_0,\; p_2=-q_0\frac{\partial\rho}{\partial
  s}\cdot\eta_0\}.$$

Note that from the choices we made, this cobordism stays in the
symplectisation of the Darboux ball. Hence one can use this to create
the local moves using Lagrangian immersion.

In order to understand the double points of this immersion note that
such a double point can only occur at $p_2=0$ which corresponds to
$\frac{\partial\rho}{\partial s}=0$ which by assumption only happens
at $s=0$. In this situation one has a double point if
$\frac{\eta_0}{3}-\rho(q,0)\eta-\frac{\partial\rho}{\partial
  q}\cdot\eta=0$. Under the assumption that the function
$q\cdot\rho(q,0)$ has only one local maximum such situation only
arises once. One can achieve this condition because as $q$ increases
the Reeb chord disappears (note that we cannot do better than
that as near $0$ this function is forced to increase because of the expansion of the cotangent direction in Equation \eqref{eq:1}).
\end{proof}

\section{Explicit Lagrangian cobordisms}
\label{sec:expl-lagr-cobord}

\subsection{A non-decomposable exact Lagrangian cobordism with compatible generating family}
\label{sec:non-decomp-exact}

We consider the following coordinates on $S^2=\{(y,p)\vert y^2+\vert p\vert^2=1\}\subset\mathbb{R}\times\mathbb{R}^2$.

Whitney's Lagrangian immersion is the map
$$\begin{array}{cccc}
  W: & S^2 &\rightarrow & T^*(\mathbb{R}_+^*\times\mathbb{R})\subset \mathbb{C}^2 \\
& (y,z) & \rightarrow & (1+iy)z.
\end{array}$$

It is an exact Lagrangian sub-manifold of $\mathbb{C}^2$ with a single transverse double point at $0$ whose Legendrian lift is the standard ``flying saucer''. Applying Theorem \ref{thm:doublepoint} one get an exact Lagrangian cobordism for $H$ to $\emptyset$ (in other word an exact concave filling of $H$). The projection to $\mathbb{R}$ of the Lagrangian cobordism has necessarily a maximum this imply that this exact Lagrangian cobordism is not decomposable in the sense of \cite{collarable} and \cite{Ekhoka}. Furthermore it is easy to see that this immersion admits a compatible generating family (as in \cite{GFcob}). The generating family is actually build similarly to the one for the standard cap of a trivial Legendrian knot appearing in \cite{collarable} (the double point arise when the variation of the length of the Reeb chords is $0$). Note that an example of a non-decomposable exact Lagrangian cobordism appears also in \cite[Section 8]{Ekhoka} referring to an example of \cite{Sauvaget}.

\subsection{Example of non-composable exact Lagrangian cobordism}
\label{sec:example-non-comp}

Using elementary cobordisms from \cite[Figure 1]{collarable} we construct an exact Lagrangian filling of the Hopf link. The local moves are described in Figure \ref{fig:hopf}. Note that topologically the cobordism is a cylinder and that the non-trivial homology class in $H_1$ is represented by a Legendrian knot (i.e. on which the Liouville form vanishes), it is thus exact.
\begin{figure}[!h]
\labellist
\small\hair 2pt
\pinlabel {$\emptyset$} [tl] at 190 -10
\endlabellist  
\centering
  \includegraphics[height=5cm]{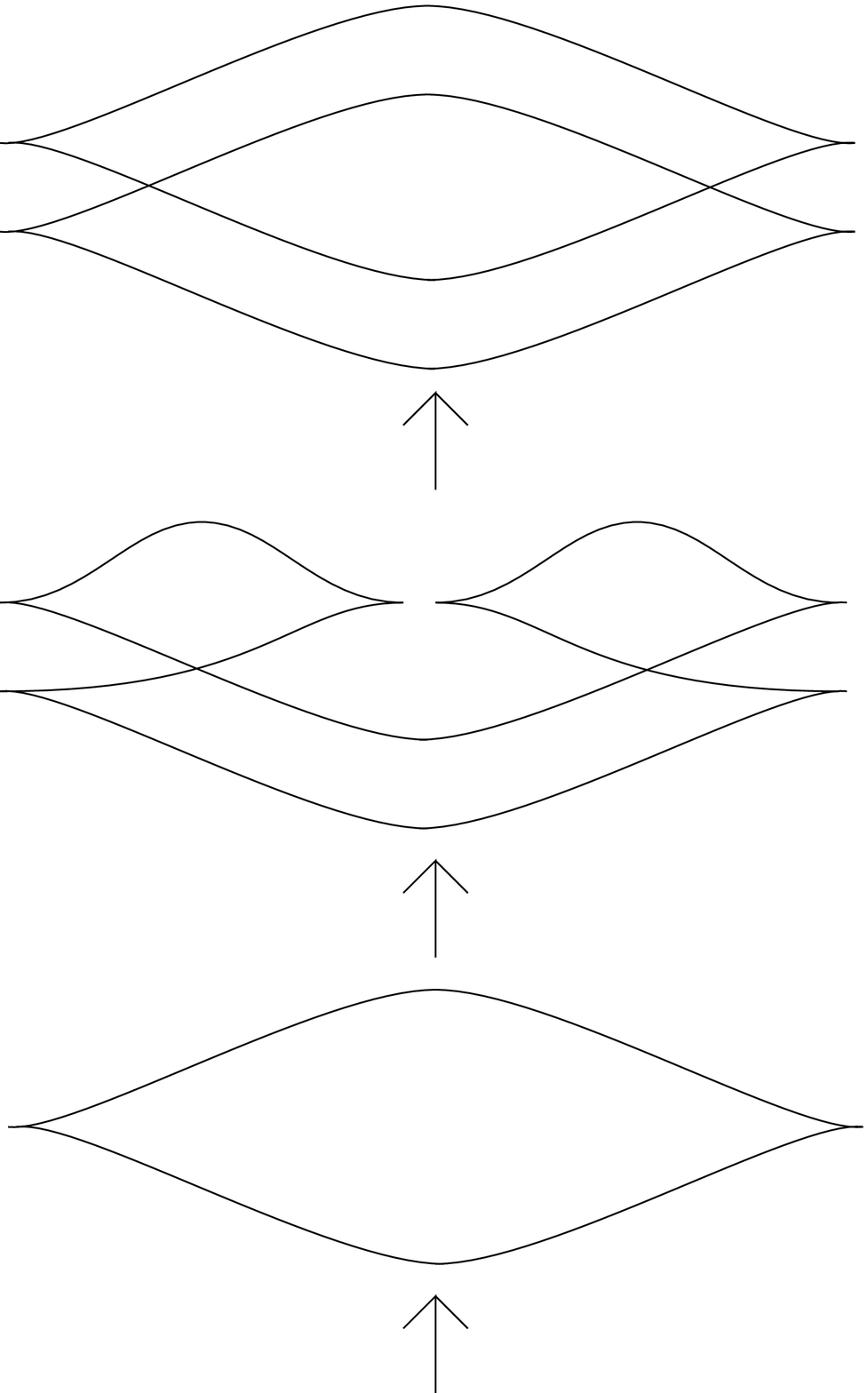}
  \caption{An exact filling of the Hopf link.}
  \label{fig:hopf}
\end{figure}

Together with the cobordism of Section \ref{sec:non-decomp-exact} we get two exact Lagrangian cobordisms: one from $\emptyset$ to $H$ and another one from $H$ to $\emptyset$. Should we be able to glue those cobordisms in an exact fashion one would get a closed exact Lagrangian torus of $\mathbb{C}^2$ which contradicts Gromov's Theorem in \cite{Gromov}.

\subsection{An exact Lagrangian cobordism to a stabilised Legendrian knot}
\label{sec:an-exact-lagrangian}

In order to construct the example of Proposition \ref{sec:prop} we will use Theorem \ref{thm:doublepoint} on a Lagrangian immersion with one double point. We construct it using the local move of Theorem \ref{thm:elementarymove}.

\begin{proof}[Proof of Proposition \ref{sec:prop}]
  Figure \ref{fig:stab} show the movie a Lagrangian immersion with one double point and empty negative ends (using move $5$ of \cite[Figure 1]{collarable}).
  \begin{figure}[!h]
\labellist
\small\hair 2pt
\pinlabel {$\emptyset$} [tl] at 220 -10
\endlabellist  
    \centering
    \includegraphics[height=5cm]{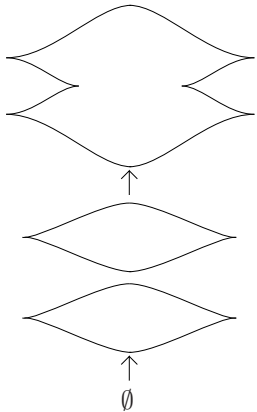}
    \caption{Movie of a Lagrangian immersion with a single double
      point.}
    \label{fig:stab}
  \end{figure}

  The domain of the immersion is topologically a disk (hence it is
  exact). As it has only one transverse double point we can apply Theorem
  \ref{thm:doublepoint}. This leads to an oriented Lagrangian
  cobordism from the Legendrian Hopf link to the stabilised Legendrian
  knot on which the Liouville form is exact. 
\end{proof}

\bibliographystyle{amsplain}
 \bibliography{Bibliographie_en}

\end{document}